\documentclass[11pt]{article}

\usepackage[a4paper, margin=1in]{geometry}
\usepackage{amsfonts,amsmath,amssymb,amsthm,graphicx,fixmath,latexsym,color}
\usepackage{xcolor}
\usepackage{hyperref,epsfig}
\usepackage{algorithmic}
\usepackage{algorithm}
\usepackage{xspace}

\def\eps{{\epsilon}}
\providecommand{\remove}[1]{}

\newcommand{\VC}{\ensuremath{\mathsf{VC}}\xspace}
\newcommand{\HD}{\ensuremath{\mathsf{HD}}\xspace}
\newcommand{\mbs}{\boldsymbol}
\newcommand{\bfm}{\textbf}

\newcommand{\mn}[0]{\medskip\noindent}
\renewcommand{\Re}{\mathbb{R}}

\newcommand{\A}{\mathcal{A}}

\newcommand{\F}{\mathcal{F}}

\newcommand{\cardin}[1]{\lvert {#1} \rvert}
\newtheorem{theorem}{Theorem}[section]
\newtheorem{lemma}[theorem]{Lemma}
\newtheorem{proposition}[theorem]{Proposition}

\newtheorem{definition}[theorem]{Definition}
\newtheorem{remark}[theorem]{Remark}
\newtheorem{notation}[theorem]{Notation}
\newtheorem*{theorem*}{Theorem}
\newtheorem*{lemma*}{Lemma}
\newtheorem*{proposition*}{Proposition}
\newtheorem{observation}[theorem]{Observation}

\begin{document}
\title{Improved bounds on the Hadwiger-Debrunner numbers\footnote{A preliminary version of this paper was presented at the SODA'2017 conference.}}

\author{Chaya Keller\thanks{Department of Mathematics, Ben-Gurion University of the NEGEV, Be'er-Sheva Israel. \texttt{kellerc@math.bgu.ac.il}. Research partially supported by Grant 635/16 from the Israel Science Foundation.}
\and
Shakhar Smorodinsky\thanks{Department of Mathematics, Ben-Gurion University of the NEGEV, Be'er-Sheva Israel and EPFL, Lausanne Switzerland. \texttt{shakhar@math.bgu.ac.il}. Research partially supported by Grant 635/16 from the Israel Science Foundation. A part of this research was carried out during the authors' visit at EPFL, supported by Swiss National Science Foundation grants 200020-162884 and 200021-165977.}
\and
G{\' a}bor Tardos\thanks{R\'enyi Institute, Budapest Hungary. \texttt{tardos@renyi.hu}. Research partially supported by the ``Lend\"ulet'' project of the Hungarian Academy of Sciences and by the National Research, Development and Innovation Office, NKFIH, projects K-116769 and SNN-117879. A part of this research was carried out during the authors' visit at EPFL, supported by Swiss National Science Foundation grants 200020-162884 and 200021-165977.}}

\date{}
\maketitle

\begin{abstract}
Let $\HD_d(p,q)$ denote the minimal size of a transversal that can always be guaranteed for a family of compact convex sets in $\Re^d$ which satisfy the $(p,q)$-property
($p \geq q \geq d+1$). In a celebrated proof of the Hadwiger-Debrunner conjecture, Alon and Kleitman proved that $\HD_d(p,q)$ exists for all $p \geq q \geq d+1$. Specifically, they prove that $\HD_d(p,d+1)$ is $\tilde{O}(p^{d^2+d})$.

We present several improved bounds:
 (i) For any $q \geq d+1$, $\HD_d(p,q) = \tilde{O}(p^{d \left(\frac{q-1}{q-d}\right)})$.
 (ii) For $q \geq \log p$, $\HD_d(p,q) = \tilde{O}(p+(p/q)^d)$.
 (iii) For every $\epsilon > 0$ there exists a $p_0 = p_0(\epsilon)$ such that for every $p \geq p_0$ and for every $q \geq p^{\frac{d-1}{d}+\epsilon}$ we have: $p-q+1 \leq \HD_d(p,q) \leq p-q+2$.
 The latter is the first near tight estimate of $\HD_d(p,q)$ for an extended range of values of $(p,q)$
since the 1957 Hadwiger-Debrunner theorem.

We also prove a $(p,2)$-theorem for families in $\Re^2$ with union complexity below a specific quadratic bound. Based on this, we introduce a polynomial time constant factor approximation algorithm for MAX-CLIQUE of intersection graphs of convex sets satisfying this property. 
\end{abstract}


\section{Introduction}
\subsection{Background}
The classical Helly's theorem says that if in a family of compact convex sets in $\Re^d$
every $d+1$ members have a non-empty intersection then the whole family has a non-empty intersection.

For a pair of positive integers $p \geq q$, we say that a family $\F$ of sets satisfies the $(p,q)$-property if $|\F|\ge p$, none of the sets in $\F$ is empty, and among any $p$ sets of $\F$ there are some $q$ with a non-empty intersection. A set $P$ is called a {\em transversal} for $\F$ if it has a non-empty intersection with every member of $\F$. In this language Helly's theorem states that any family of compact convex sets in $\Re^d$ satisfying the $(d+1,d+1)$-property has a singleton transversal.
In an attempt to generalize Helly's theorem, Hadwiger and Debrunner \cite{HD57} posed a conjecture that was proved more than 30 years later in a celebrated result of Alon and Kleitman:
\begin{theorem}[the $(p,q)$-theorem \cite{AK,AK97}]
For any triple of positive integers $p \geq q \geq d+1$ there exists an integer $s$ such that if $\F$ is a family of compact convex sets in $\Re^d$ satisfying the $(p,q)$-property, then there exists a transversal for $\F$ of size at most $s$.
\end{theorem}

We denote the smallest value $s$ that works for $p\ge q>d$ by $\HD_d(p,q)$.

The $(p,q)$-theorem has a rich history of variations and generalizations described in the survey of Eckhoff \cite{ECK03}.
Those include a version for set systems with bounded \VC-dimension \cite{MAT04}, colorful and fractional versions \cite{BFMOP14} and a generalization to a topological $(p,q)$-theorem for finite families of sets which are so-called {\em good cover}, i.e., the intersection of every sub-family is either empty or contractible \cite{AKMM}.

The upper bound on $\HD_d(p,q)$ provided in the Alon-Kleitman proof \cite{AK} is huge and it is believed that much better bounds could be achieved. In fact, Alon and Kleitman were only interested in proving the existence of $\HD_d(p,q)$ and hence concentrated on the case $q=d+1$. Their bound is $\HD_d(p,d+1) = \tilde{O}(p^{d^2+d})$ where $\tilde{O}$ hides some polylogarithmic factors.
They write: ``Although the proof supplies finite upper bounds for $\HD_d(p,q)$,\footnote{Alon and Kleitman use the notation $M(p,q,d)$.} the bounds obtained are very large and the problem of determining this function precisely remains wide open.''
In fact, it is not clear how their method  can improve the asymptotic of $\HD_d(p,q)$ when $q$ is slightly more than $d+1$, say $2d$. In their second paper on the subject \cite{AK97} Alon and Kleitman provide a more elementary proof of the same result that gives slightly weaker bounds.

Trivially, for any $p \geq q$ we have $\HD_d(p,q) \geq p-q+1$. Hadwiger and Debrunner \cite{HD57} proved the following:
\begin{theorem}[\cite{HD57}]\label{thm:HD-thm}
For  $p \geq q \geq d+1$ such that $q > \frac{d-1}{d}p +1$
$$HD_d(p,q) = p-q+1.$$
\end{theorem}
The precise bound is not known already in the plane when $p=4$ and $q=3$. The best known upper and lower bounds in that case are due to Kleitman et al. \cite{KGT01}
who showed:
$$
3 \leq \HD_2(4,3) \leq 13,
$$
improving the upper bound of $345$ obtained in \cite{AK} for that special case.
The best known general lower bound is
\[
\HD_d(p,q) = \Omega \left( \frac{p}{q} \log^{d-1} \frac{p}{q} \right),
\]
that follows easily from a lower bound construction for weak $\eps$-nets due to Bukh et al.~\cite{BMN11}.

Matou\v{s}ek~\cite{MATOUSEK} writes that the Hadwiger-Debrunner bound in Theorem~\ref{thm:HD-thm} ``is the only nontrivial case where exact values, or even good estimates of $\HD_d(p,q)$, are known''.

%
%
%

\subsection{Improved bounds on $\HD_d(p,q)$}
In this paper we improve the asymptotic bounds on $\HD_d(p,q)$. We think of the dimension $d$ as a fixed constant and are interested in $\HD_d(p,q)$ as a function of $p,q$. Accordingly, the notation $O(\cdot)$ may hide dependence on $d$. Additionally, in the notation $\tilde O(\cdot)$ we also supress polylogarithmic factors in $p$. For $d \geq 3$, our main result is the following.

\begin{theorem}\label{Thm:Main}
For $p \geq q >d\ge3 $ and $\eps>0$ the Hadwiger-Debrunner numbers $\HD_d(p,q)$ satisfy:
$$
\HD_d(p,q) \leq \begin{cases}
\mathrm{(a)} \quad O\left(p^{d \cdot \frac{q-1}{q-d}} \log^{cd^3 \log d} p\right) = \tilde{O} \left(p^{d \cdot \frac{q-1}{q-d}} \right) ;\\
\mathrm{(b)} \quad p-q+ O\left(\left(\frac{p}{q}\right)^d \log^{cd^3 \log d} \left(\frac{p}{q} \right) \right) =
\tilde{O} \left(p+ \left(\frac{p}{q}\right)^d \right) & \textrm{ if } q \geq \log p;\\
\mathrm{(c)} \quad p-q+2 & \textrm{ if } q \ge p^{\frac{d-1}{d}+\eps}, p \geq p_d(\eps);\\
\end{cases}$$
\noindent where $c$ is an absolute constant and the threshold $p_d(\eps)$ depends on $d$ and $\eps$.
\end{theorem}

\noindent For $d=2$, parts (a) and (b) of our result are a bit stronger.

\begin{theorem}\label{Thm:Main2}
For $p \geq q \geq 3$ and $\eps>0$ the Hadwiger-Debrunner numbers $\HD_2(p,q)$ satisfy:
$$
\HD_2(p,q) \leq \begin{cases}
\mathrm{(a)} \quad O \left(p^{2 \cdot \frac{q-1}{q-2}} \right);\\
\mathrm{(b)} \quad p-q+ O\left(\left(\frac{p}{q}\right)^2 \log^{2} \left(\frac{p}{q} \right) \right) & \textrm{ if } q \geq \log p;\\
\mathrm{(c)} \quad p-q+2 & \textrm{ if } q \ge p^{\frac{1}{2}+\eps}, p \geq p_2(\eps);\\
\end{cases}$$
\noindent where the threshold $p_2(\eps)$ depends on $\eps$.
\end{theorem}

We note that already Case~(a) provides improved bounds over the one obtained by Alon and Kleitman.
Case~(b) represents a significant improvement and one cannot improve this bound further significantly without also improving the results for the well studied problem of weak $\eps$-nets for convex sets (see Theorem~\ref{Thm:Weakepsnet} and the remarks in Section~\ref{sec:discussion} for more details).
Case~(c) is an extension of the Hadwiger-Debrunner tight bounds to the wider range of values $q \geq p^{\frac{d-1}{d}+\eps}$ rather than $q > \frac{d-1}{d}p +1$.

\mn The proof of~(a) follows the Alon-Kleitman proof of the $(p,q)$-theorem, and the improvement is obtained by replacing two steps of the proof with a classical hypergraph Tur\'{a}n-type result of de Caen~\cite{dC83} and a tight form of the Upper Bound Theorem for convex sets proved by Kalai~\cite{Kalai84}.
The proof of~(b) is an inductive bootstrapping process that exploits the result of~(a). The proof of~(c) is yet another bootstrapping, using~(a), (b), and the Hadwiger-Debrunner theorem. Both of these bootstrapping arguments are based on the following dichotomy:
\begin{observation}\label{obs}
Assume that $\F$ satisfies the $(p,q)$-property. For any $p'<p$, $q'<q$, either $\F$ satisfies the $(p',q')$ property, or there exists a sub-family $S \subset \F$ with $p'$ elements, and with no $q'$ intersecting elements. In the latter case, $\F \setminus S$ satisfies the $(p-p',q-q'+1)$ property.
\end{observation}

\subsection{ A $(p,2)$-theorem in the plane for sets with union complexity below a quadratic bound}

The finiteness of $\HD_d(p,q)$ is only proved for $q \geq d+1$. The transversal number of a family $\F$ of compact convex sets in $\Re^d$ satisfying the $(p,d)$-property is not bounded as a function of $p$ and $d$, even if $p=d$. This is easily seen (as noted in \cite{AK}) by taking a family $\F$ of $n$ hyperplanes (for an arbitrary large $n$) in general position in $\Re^d$. To make those sets compact we intersect all those hyperplanes with a box containing all the $n \choose d$ intersection points of all $d$-tuples of hyperplanes. Obviously, $\F$ satisfies the $(d,d)$-property but no transversal for $\F$ has size less than $n/d$.

In some special cases it is known that a $(p,2)$-theorem in the plane does exist. It is well known that every family of pseudo-discs satisfying the $(p,2)$-property admits a transversal of size $O(p)$ (see, e.g., \cite{CH12,PR08}). Danzer \cite{Danzer86} proved that a family of pairwise intersecting discs (i.e., a family of discs satisfying the $(2,2)$-property) in $\Re^2$ admits a transversal of size four. Let $\gamma$ be a convex curve in the plane.
Recently, Govindarajan and Nivasch \cite{GN15} proved a $(p,2)$-theorem when the intersections of pairs belong to $\gamma$. Specifically, they prove that if for a family $\F$ of compact convex sets in the plane, their intersections with $\gamma$ satisfy the $(p,2)$-property, then $\F$ has a transversal of size $O(p^8)$.

Families with bounded union complexity were also considered in connection with $(p,2)$-theorems. Here we use the following definition.
\begin{definition}
Let $\F$ be a family of $n$ simple Jordan regions in the plane.
The {\em union complexity} of $\F$ is the number of vertices
(i.e., intersection of boundaries of pairs of regions in $\F$)
that lie on the boundary $\partial \bigcup_{r \in \F} r$.
\end{definition}
The notion of union complexity has been the subject of many papers. Researchers were interested in bounding the union complexity of various families of objects and understanding other combinatorial properties of families with ``low'' union complexity. See, e.g., the survey of Agarwal et al. \cite{APS-union}.
It is known that the union-complexity of any $n$ discs (or even pseudo-discs) is at most $6n-12$ \cite{KLPS}.

The results of Pinchasi \cite{P15} imply that if the union complexity of $n$ elements of a planar family $\F$ is sub-quadratic in $n$, then the fractional Helly number of the family is $2$. This, combined with the techniques of Alon and Kleitman, implies a $(p,2)$ theorem for compact convex sets in the plane with sub-quadratic union complexity.

Here we prove a $(p,2)$ theorem for compact convex sets in the plane with a somewhat weaker bound on the union complexity. In some combinatorial sense, it shows that the only counter example to the finiteness of $\HD_2(2,2)$ is given by families which ``resemble'' lines.
\begin{theorem}\label{Thm:(p,2)}
Let $\F$ be a family of compact convex sets in the plane satisfying the $(p,2)$ property. Assume that for some (fixed) $k\ge3$ the union complexity of every $k$ sets from $\F$ is less than $\binom k2$. Then $\F$ admits a transversal of size $O(k^4p^{16})$.
\end{theorem}

%
%
%
%

\subsection{Approximating the clique number for intersection graphs}
Let $\F$ be a finite family of sets. The intersection graph $G(\F)$ is the graph $(\F,E)$ where $E$ consists of all pairs of sets in $\F$ with a non-empty intersection.
The computational complexity of the maximum-clique problem in intersection graphs of discs is not known. In particular, it is not known whether it is NP-hard.
The best known polynomial time algorithm gives a $2$-approximation factor. That is, it finds a subset of the discs that forms a clique in the intersection graph whose size is at least $opt/2$, where $opt$ is the size of the maximum clique. Amb{\"{u}}hl and Wagner \cite{Wagner-clique} proved that the MAX-CLIQUE problem for families of fat ellipses is APX-HARD.

Let $\F$ be a family of convex sets in the plane satisfying the conditions of Theorem~\ref{Thm:(p,2)}. As a corollary of our $(p,2)$-theorem we obtain a simple polynomial time algorithm which approximates the maximum-clique for the intersection graph of any finite subfamily of $\F$ within a constant factor $C$ depending only on the family $\F$. We note that there is no hope to find a PTAS (polynomial-time approximation scheme) for such families. Indeed, this follows from the hardness result of Amb{\"{u}}hl and Wagner for fat ellipses, combined with the fact that fat ellipses have sub-quadratic union complexity.

\subsection{Organization of the paper}

The paper is organized as follows: In Section~\ref{sec:main} we prove Theorems~\ref{Thm:Main} and~\ref{Thm:Main2}.
In Section~\ref{sec:(p,2)} we prove
Theorem~\ref{Thm:(p,2)} and present the approximation algorithm for the clique number of certain intersection graphs. We conclude the paper with a discussion
in Section~\ref{sec:discussion}.

\section{Proof of the Main Theorem}
\label{sec:main}

In this section we present the proof of Theorems~\ref{Thm:Main} and \ref{Thm:Main2}. Since the proofs of cases~(a), (b), and (c)  use different methods, we present each of them in a separate subsection. 

\subsection{Improved bound on $\HD_d(p,q)$ for any $q \geq d+1$}

In this subsection we prove Theorem~\ref{Thm:Main}(a), namely, that $\HD_d(p,q) \leq \tilde{O}(p^{d\left(\frac{q-1}{q-d}\right)})$.
By compactness, it is enough to provide a bound on $\HD_d(p,q)$ for finite families of convex sets.
Our proof follows some steps of the Alon-Kleitman proof of the $(p,q)$-theorem.

Let $\F$ be a family of $n$ compact convex sets in $\Re^d$ that satisfies the $(p,q)$-property. The Alon-Kleitman proof consists of the
following four steps:
\begin{enumerate}
\item
Count the number of $(d+1)$-tuples of sets in $\F$ with a non-empty intersection. Using a double-counting argument there are
$\Omega \left (\frac{n^{d+1}}{p^{d+1}} \right )$ such tuples.

\item
Apply the {\em Fractional Helly Theorem} (first proved by Katchalski and Liu in \cite{KL79}, see also \cite{Kalai84,MATOUSEK}) to conclude that there is a point that pierces at least $\Omega(\frac{n}{p^{d+1}})$ of the sets.

\item
Use the Linear-Programming duality to show that there is a finite weighted set $P$ of points with a total weight $W$ such that every member in $\F$ contains a subset of $P$ of total weight $\Omega(\frac{W}{p^{d+1}})$.

\item
Apply known bounds for weak $\eps$-nets (see, e.g., \cite{MATOUSEK}) with $\eps = \Omega(\frac{1}{p^{d+1}})$ to show that $\F$ can be pierced with $f(\eps,d)= \tilde{O}(\frac{1}{\eps^d}) = \tilde{O}(p^{(d+1)d})$ points.
\end{enumerate}

To obtain a better bound for $\HD_d(p,q)$, we replace the direct arguments in the first two steps of the proof with stronger and deeper tools. In particular, for the first step we use the following Tur\'{a}n-type result for hypergraphs,
proved by de Caen~\cite{dC83} (see also~\cite{Keevash}). We note that a slightly weaker result can be proved by a simple
probabilistic argument.

\begin{theorem}[de Caen, 1983]\label{Thm:dC}
Let $n \geq p \geq q$. Let $\mathcal{H}$ be a $q$-uniform hypergraph on $n$ vertices that does not contain an independent set of
size $p$. Then
\[
|E(\mathcal{H})| \geq \frac{n-p+1}{n-q+1} \cdot \frac{{{n}\choose{q}}}{{{p-1}\choose{q-1}}}.
\]
\end{theorem}

For the second step we use Kalai's tight form of the Upper Bound Theorem for convex sets (\cite{Kalai84}, see also~\cite{AlonKalai,Eckhoff}).

\begin{theorem}[Kalai, 1984]\label{Thm:Kalai}
Let $\F$ be a family of $n$ convex sets in $\Re^d$. Denote by $f_{k-1}$ the number of $k$-tuples of sets in $\F$
whose intersection is non-empty. If $f_{d+r}=0$ for some $r \geq 0$ then for any $k>0$,
\[
f_{k-1} \leq \sum_{i=0}^d {r \choose k-i} {n-r \choose i}.
\]
\end{theorem}

Combining these results we establish:

\begin{proposition}\label{Prop1}
Let $\F$ be a family of $n \geq2p$ compact convex sets in $\mathbb{R}^d$ that satisfies the $(p,q)$-property,
$p \geq q \geq d+1$. Then there exists a point that pierces at least
$\Omega \left( \frac{qn}{p^{\frac{q-1}{q-d}}} \right)$ elements of $\F$.
\end{proposition}

\begin{proof}
Denote by $x$ the number of $q$-tuples of sets in $\F$ with a non-empty
intersection, and assume that there is no point that pierces more than $m=\alpha n$ of the sets in $\F$. We have
\begin{equation}\label{Eq1}
\frac{n-p+1}{n-q+1} \cdot \frac{{{n}\choose{q}}}{{{p-1}\choose{q-1}}} \leq x \leq \sum_{i=0}^d {{m-d}\choose{q-i}} {{n-(m-d)}\choose{i}} .
\end{equation}
The left inequality in~\eqref{Eq1} follows from Theorem~\ref{Thm:dC} (applied to the hypergraph whose vertices are
the elements of $\F$ and whose edges are $q$-tuples whose intersection is non-empty) and the right inequality follows from
Theorem~\ref{Thm:Kalai} (applied with $r=m-d, k=q$). As $n \geq2p$ by assumption, we have
\[
\frac{n^q}{2q p^{q-1}} \leq \frac{n-p+1}{n-q+1} \cdot \frac{{{n}\choose{q}}}{{{p-1}\choose{q-1}}}.
\]
Hence,~\eqref{Eq1} implies
\[
\frac{cn^q}{q p^{q-1}} \leq \sum_{i=0}^d n^i \frac{(\alpha n)^{q-i}}{(q-i)!} \leq \frac{n^q}{(q-d)!}
\sum_{i=0}^d \alpha^{q-i}.
\]
Assuming $\alpha<1/2$ (since otherwise there is a point that stabs $n/2$ elements of $\F$) and using Stirling's formula, we get
\[
\frac{(q-d)^{q-d}}{4q  e^{q-d} p^{q-1}} \leq \alpha^{q-d},
\]
which implies $\alpha = \Omega \left(\frac{q}{p^{\frac{q-1}{q-d}}} \right)$, as asserted.
\end{proof}

The rest of the proof of Theorem~\ref{Thm:Main}(a) follows steps~(3)-(4) of Alon-Kleitman's proof.
Two classical results are needed.

\medskip The first is an LP duality lemma proved (implicitly) by Alon and Kleitman using a well known variant of Farkas' Lemma (cf. \cite{MATOUSEK}).
\begin{lemma}\label{duality}
Let $0 < \alpha < 1$ be a fixed real number. Let $\F$ be a finite family of sets.
Suppose that for any multiset $\F'$ consisting of elements of $\F$ there exists a point $x$ that is contained in at least $\alpha \cardin{\F'}$ members of $\F'$. Then there exists a finite multiset $P$ of points such that every member of $\F$ contains at least $\alpha \cardin{P}$ elements of $P$.
\end{lemma}

The second is a bound on the size of weak $\eps$-nets:
\begin{theorem}[weak $\eps$-nets \cite{ABFK,CEGGSW,MW}]\label{Thm:Weakepsnet}
For every real $0 < \eps < 1$ and for every integer $d$ there exists a constant $f = f(\eps,d)$ such that the following holds: For every $n$ and for every multiset $P$ of $n$ points in $\Re^d$, there exists a set $N$ of at most $f(\eps,d)$ points such that every convex set containing at least $\eps \cardin{P}$ points of $P$ must also contain a point of $N$.
\end{theorem}
The finiteness of $f(\eps,d)$ was first proved by Alon et al. \cite{ABFK} and better bounds were obtained by Chazelle at al. in \cite{CEGGSW}. The current best known upper bound due to Matou{\v s}ek and Wagner \cite{MW} is
$f(\eps,d)= O(\frac{1}{\eps^d} \log^{c(d)} \frac{1}{\eps})$, where $c(d) = O(d^3 \log d)$ for $d \geq 3$ and
$f(\eps,2) = O(\frac{1}{\eps^2})$ \cite{ABFK}. The best known lower
bound was provided by Bukh, Matou{\v s}ek and
Nivasch \cite{BMN11} who showed that $f(2,\eps) = \Omega(\frac{1}{\eps}\log\frac{1}{\eps})$ and for general $d \geq 3$, $f(d,\eps) =\Omega(\frac{1}{\eps}(\log \frac{1}{\eps})^{d-1})$. It remains a big open problem to provide sharp bounds on $f(\eps,d)$.

Part (a) of Theorems~\ref{Thm:Main} and \ref{Thm:Main2} is obtained by plugging in the best known bounds for Theorem~\ref{Thm:Weakepsnet} in the following result:

\begin{proposition}\label{propa}
For $p\ge q>d\ge2$ the Hadwiger-Debrunner numbers $\HD_d(p,q)$ satisfy $\HD_d(p,q)\le f(\beta,d)$, where $f$ is the function from Theorem~\ref{Thm:Weakepsnet} and $\beta=\Omega\left(p^{-\frac{q-1}{q-d}}\right)$.
\end{proposition}

\begin{proof}
Let $\F$ be a family that satisfies the assumption of the theorem. Let $\F'$ be a multiset of elements of $\F$. If $|\F'|\ge p':=(p-1)(q-1)+1$, then it satisfies the $(p',q)$ property as among among $p'$ sets we either find $p$ distinct sets or $q$ copies of the same set. If $|\F'|\ge2p'$, then Proposition~\ref{Prop1} implies
that there exists a point that pierces at least $\beta |\F'|$ elements of $\F'$, for
\[
\beta= \Omega \left(\frac{q}{p'^{\frac{q-1}{q-d}}} \right) = \Omega \left(\frac{1}{p^{\frac{q-1}{q-d}}} \right),
\]
where the second equality holds since $p'<pq$ and $q^{(q-1)/(q-d)-1}=q^{(d-1)/(q-d)}$ is bounded from above by a constant depending only on $d$. The existence of a point piercing $\beta|\F'|$ elements of $\F'$ remains true even for smaller multisets. This is so because the multiplicities of the sets in $\F'$ can be multiplied to increase the size of the family but this does not affect the ratio a point pierces. By Lemma~\ref{duality} it follows that there exists a finite multiset $P$ such that each element of $\F$ contains at least
$\beta |P|$ points of $P$. Hence, by Theorem~\ref{Thm:Weakepsnet}, $\F$ admits a
transversal of size $f(\beta,d)$.
\end{proof}

\begin{remark}\label{rem}
We note that when $q=\Omega(\log p)$, then in Proposition~\ref{propa} we have $\beta=\Omega(1/p)$ and $\HD_d(p,q) = \tilde O(p^d)$. In the next subsection we improve this bound to $\tilde O\left((p/q)^d\right)$.
\end{remark}

\subsection{Improved bound on $\HD_d(p,q)$ for $q \geq \log p$}

In this subsection we prove part (b) of Theorems~\ref{Thm:Main} and \ref{Thm:Main2}. We prove the slightly stronger statement below. Note that for $\beta>(q-1)/p$ and an arbitrary multiset of points $P$, the family of sets containing at least $\beta|P|$ points of $|P|$ satisfy the $(p,q)$-property. Thus, $\HD_d(p,q)$ can work as the upper bound $f(q/p,d)$ in Theorem~\ref{Thm:Weakepsnet}. We also have $p-q<\HD_d(p,q)$, and therefore our bound here is almost optimal except for the logarithmic term in $\beta$.

\begin{proposition}\label{Prop:q>log p}
Let $p\ge q \geq d+1$ such that $q \geq \log p$. Then
\[
\HD_d(p,q) \leq p-q +f(\beta,d),
\]
where $f$ is the function from Theorem~\ref{Thm:Weakepsnet} and $\beta=\Omega\left((\frac pq\log\frac pq)^{-1}\right)$.
\end{proposition}

\begin{proof}  Let $\F$ be a family of compact convex sets in $\Re^d$ that satisfies the $(p,q)$-property for $p\ge q>d$ and $q\ge\log p$.
Put $k = \max(\lceil\log(p/q)\rceil,d)$ and $k'=\lceil kp/q\rceil$. For $\ell \geq 0$, define
\[
p_\ell = p- \ell k' \qquad \mbox{ and } \qquad q_\ell = q-\ell k.
\]
Note that $k'/k\ge p/q$ and therefore $p_\ell/q_\ell\le p/q$.
Find the largest $\ell$ such that $q_\ell>k$ and $\F$ satisfies the $(p_\ell,q_\ell)$-property. Surely $\ell=0$ satisfies both requirements, so such a largest $\ell$ exists. We consider two cases according to which requirement $\ell+1$ violates.

If $q_{\ell+1}\le k$, then $q_\ell=q_{\ell+1}+k\le2k$  and so $p_\ell\le(p/q)q_\ell=O((p/q)\log(p/q))$. We also have $q_\ell>k=\Omega(\log p_\ell)$ and therefore Remark~\ref{rem} applies and $\F$ has a transversal of size $\HD_d(p_\ell,q_\ell)\le f(\beta,d)$ with $\beta=\Omega(1/p_\ell)=\Omega\left((\frac pq\log\frac pq)^{-1}\right)$.

If $\F$ does not satisfy the $(p_{\ell+1},q_{\ell+1})$ property, then we apply Observation~\ref{obs}. We find a subset $S\subseteq\F$ with $|S|=p_{\ell+1}$ such that no $q_{\ell+1}$ of them intersect and conclude that $\F\setminus S$ satisfies the $(k',k+1)$-condition. We can apply Remark~\ref{rem} again to see that $\F\setminus S$ can be pierced by $\HD_d(k',k+1)\le f(\beta,d)$ points, where $\beta=\Omega(1/k')=\Omega\left((\frac pq\log\frac pq)^{-1}\right)$. Finally $S$ (as any subset of $\F$ of size at most $p$) can be pierced by $p-q+1$ points. This finishes the proof of the proposition.
\end{proof}

\subsection{Improved bound on $\HD_d(p,q)$ for $q \geq p^{1-\frac{1}{d}+\eps}$}

In this subsection we prove part (c) of Theorems~\ref{Thm:Main} and \ref{Thm:Main2}. The proof consists of three steps:
\begin{enumerate}
\item First, we prove a weaker version in which we replace the requirement $q\ge p^{1-\frac{1}{d}+\eps}$ with the slightly stronger requirement $q\ge p^{1-\frac{1}{d+1}+\eps}$ to obtain the same conclusion. This step is established in Proposition~\ref{Prop:Large-q} below.

\item Second, we prove another weaker version in which the requirement $q\ge p^{1-\frac{1}{d}+\eps}$ of the theorem is preserved, but the conclusion is weakened to piercing with $p-q+O(\log_d(1/\eps))$ points (instead of $p-q+2$ points in the statement of the theorem). This step is established in Proposition~\ref{Prop:Semi-Large-q} by an inductive bootstrapping argument, with Proposition~\ref{Prop:Large-q} as its basis.

\item Finally, we prove the full statement of part (c) by another bootstrapping argument, combining the results of Steps~1 and~2.
\end{enumerate}

\begin{proposition}\label{Prop:Large-q}
For any $d$ and $\eps>0$, there exists $p_d(\eps)$ such that for all $p>p_d(\eps)$ and all $q>p^{1-\frac{1}{d+1}+\eps}$, we have
$\HD_d(p,q) \leq p-q+2$.
\end{proposition}

\begin{proof}
Let $p,q$ satisfy the assumptions with $p$ large enough, and let $\F$ be a family of compact convex sets
in $\Re^d$ that satisfies the $(p,q)$-property. We use Observation~\ref{obs} to distinguish two cases:

\mn \textbf{Case 1: $\F$ satisfies the $(p - \lfloor \frac{q}{d-1} \rfloor, (d-1) \lceil \frac{q}{d} \rceil -d)$-property.}
By Proposition~\ref{Prop:q>log p} (whose assumption is clearly satisfied by $\F$, when $p$ is large enough), this implies that $\F$ has a transversal of size
\begin{equation}\label{Eq:Aux1}
(p - \lfloor \frac{q}{d-1} \rfloor) - ((d-1) \lceil \frac{q}{d} \rceil -d) +
O \left( \left(\frac{p - \lfloor \frac{q}{d-1} \rfloor}{(d-1) \lceil \frac{q}{d} \rceil -d} \right)^d \log^{cd^3 \log d}
\left(\frac{p - \lfloor \frac{q}{d-1} \rfloor}{(d-1) \lceil \frac{q}{d} \rceil -d} \right) \right),
\end{equation}
where $c$ is a universal constant. Now, we have
\[
(p - \lfloor \frac{q}{d-1} \rfloor) - ((d-1) \lceil \frac{q}{d} \rceil -d) \leq p- \frac{q}{d-1} +1 -
(d-1)\frac{q}{d}+d = p-q- \frac{q}{d(d-1)} +d+1.
\]
Hence, if we show that the $O(\cdot)$ term in~\eqref{Eq:Aux1} is negligible with respect to $\frac{q}{d(d-1)}$, it will follow that
for a sufficiently large $p$, $\F$ has a transversal of size less than $p-q$. And indeed, as $q\ge p^{\frac{d}{d+1}+\eps}$,
we have
\begin{align*}
O \left( \left (\frac{p - \lfloor \frac{q}{d-1} \rfloor}{(d-1) \lceil \frac{q}{d} \rceil -d} \right)^d \log^{cd^3 \log d}
\left(\frac{p - \lfloor \frac{q}{d-1} \rfloor}{(d-1) \lceil \frac{q}{d} \rceil -d} \right) \right) &\leq
O \left( \left( \frac{p}{q} \right)^d \log^{cd^3 \log d} \left( \frac{p}{q} \right) \right) \\
&= \tilde{O} \left( p^{\frac{d}{d+1}-d\eps} \right) = o \left( \frac{q}{d(d-1)} \right),
\end{align*}
as asserted.

\mn \textbf{Case 2: $\F$ contains a sub-family $S$ of size $p - \lfloor \frac{q}{d-1} \rfloor$ without
an intersecting $((d-1) \lceil \frac{q}{d} \rceil -d)$-tuple.} Denote the maximal size of an intersecting sub-family
of $S$ by $(d-1) \lceil \frac{q}{d} \rceil -t$, for $t>d$. In such a case, $\F \setminus S$ satisfies the
$\left(\lfloor \frac{q}{d-1} \rfloor, q - (d-1) \lceil \frac{q}{d} \rceil +t \right)$ property.
We claim that these parameters satisfy the condition of Theorem~\ref{thm:HD-thm}. In our case, the condition reads
\[
(d-1)\lfloor \frac{q}{d-1} \rfloor < d(q - (d-1) \lceil \frac{q}{d} \rceil +t-1).
\]
Thus, it is sufficient to show that $q< qd -d(d-1) (\frac{q}{d}+1) +dt-d$. And indeed, we have
\[
qd -d(d-1) (\frac{q}{d}+1) +dt-d = qd-q(d-1)-d(d-1)+dt-d = q + d(t-d) > q,
\]
where the last inequality holds since $t>d$. Therefore, by the Hadwiger-Debrunner theorem, $\F \setminus S$
can be pierced by
\[
\lfloor \frac{q}{d-1} \rfloor - \left( q - (d-1) \lceil \frac{q}{d} \rceil +t \right) +1
\]
points. As $S$ can clearly be pierced by
\[
(p - \lfloor \frac{q}{d-1} \rfloor) - ((d-1) \lceil \frac{q}{d} \rceil -t) +1
\]
points (by piercing its maximal intersecting subfamily by a single point and each other element by a
separate point), $\F$ has a transversal of size
\[
\left(\lfloor \frac{q}{d-1} \rfloor - \left( q - (d-1) \lceil \frac{q}{d} \rceil +t \right) +1 \right) +
\left((p - \lfloor \frac{q}{d-1} \rfloor) - ((d-1) \lceil \frac{q}{d} \rceil -t) +1 \right) = p-q+2.
\]
This completes the proof of the proposition.
\end{proof}

For the following propositions, we need an additional notation.
\begin{notation}
We say that $(p,q)$ are \emph{$(k,\hat{\eps})$-close} if $q> p^{\frac{d^k(d-1)}{d^{k+1}-1}+\hat{\eps}}$.
\end{notation}

\begin{proposition}\label{Prop:Semi-Large-q}
For any $\hat{\eps}>0$ and any $k \in \mathbb{N}$, there exists $p_2(\hat{\eps},k)$ such that for all $(p,q)$ that are $(k,\hat{\eps})$-close with $p>p_2$, we have $\HD_d(p,q) \leq p-q+(k+1)$.
\end{proposition}

\noindent Note that the proposition implies that if $p$ is sufficiently large and $q\ge p^{1-\frac{1}{d}+\eps}$ then $\HD_d(p,q) \leq p-q+O(\log_d(1/\eps))$. We do not prove this implication, as in the sequel we will use the proposition itself rather than this corollary.

\begin{proof}
For the sake of simplicity, we assume that all quotients that we consider are integers. It is easily seen that this is without loss of generality.

\mn As for any $q'<q$ we have $\HD_d(p,q) \leq \HD_d(p,q')$, it is sufficient to prove that for any sufficiently large $p$ and any
\begin{equation}\label{Eq:Semi-large1}
p^{\frac{d^k(d-1)}{d^{k+1}-1}+\hat{\eps}} < q \leq p^{\frac{d^{k-1}(d-1)}{d^{k}-1}}
\end{equation}
we have $\HD_d(p,q) \leq p-q+(k+1)$.

We prove this by induction on $k$. The case $k=1$ is exactly the assertion of Proposition~\ref{Prop:Large-q}. Assume that the assertion holds for some $k \geq 1$, and let $\F$ be a family of compact convex sets in $\Re^d$ that satisfies the $(p,q)$-property, where
\begin{equation}\label{Eq:Semi-large2}
p^{\frac{d^{k+1}(d-1)}{d^{k+2}-1}+\hat{\eps}} < q \leq p^{\frac{d^{k}(d-1)}{d^{k+1}-1}}.
\end{equation}
We consider two cases.

\mn \textbf{Case 1: $\F$ satisfies the
\[
(p - q^{\frac{(1-\lambda)(d^{k+1}-1)}{d^k(d-1)}}, q-q^{1-\lambda/2})
\]
property, for a sufficiently small $\lambda=\lambda(\hat{\eps})$ to be determined below.}
By Theorem~\ref{Thm:Main}(b) (whose assumption is clearly satisfied by $\F$), this implies that $\F$ has a transversal of size
\begin{equation}\label{Eq:Semi-large3}
p - q^{\frac{(1-\lambda)(d^{k+1}-1)}{d^k(d-1)}} - q + q^{1-\lambda/2} + \tilde{O} \left( \left(\frac{p - q^{\frac{(1-\lambda)(d^{k+1}-1)}{d^k(d-1)}}}
{q-q^{1-\lambda/2}} \right)^d \right).
\end{equation}
For a sufficiently small $\lambda$ (as function of $d,k$) we have:
\begin{enumerate}
\item $q^{1-\lambda/2} \ll q^{\frac{(1-\lambda)(d^{k+1}-1)}{d^k(d-1)}}$,

\item $q^{1-\lambda/2} \ll q$, and

\item $q^{\frac{(1-\lambda)(d^{k+1}-1)}{d^k(d-1)}} \ll p$.
\end{enumerate}
Hence, if we show that
\begin{equation}\label{Eq:Semi-large4}
(p/q)^d \ll q^{\frac{(1-\lambda)(d^{k+1}-1)}{d^k(d-1)}},
\end{equation}
it would follow that the $\tilde{O}(\cdot)$ term in~\eqref{Eq:Semi-large3} is asymptotically negligible with respect to
\[
q^{\frac{(1-\lambda)(d^{k+1}-1)}{d^k(d-1)}} - q^{1-\lambda/2},
\]
and hence, for a sufficiently large $p$, $\F$ has a transversal of size less than $p-q+2$. To see that Equation~\eqref{Eq:Semi-large4}
holds, note that by~\eqref{Eq:Semi-large2},
\begin{align*}
\frac{p^d}{q^d q^{\frac{(1-\lambda)(d^{k+1}-1)}{d^k(d-1)}}} \leq \frac{q^{\frac{d}{\frac{d^{k+1}(d-1)}{d^{k+2}-1}+\hat{\eps}}}}
{q^{d+(1-\lambda)\frac{d^{k+1}-1}{d^k(d-1)}}} = \frac{q^{\frac{d}{\frac{d^{k+1}(d-1)}{d^{k+2}-1}+\hat{\eps}}}}
{q^{\frac{d^{k+2}-1}{d^k(d-1)}- \lambda \frac{d^{k+1}-1}{d^k(d-1)}}}.
\end{align*}
In the last expression, the exponent of $q$ in the numerator is smaller than the first term $\frac{d^{k+2}-1}{d^k(d-1)}$ of the exponent
in the denominator. Hence, for $\lambda$ sufficiently small as function of $\hat{\eps},d,k$, the overall exponent of $q$ in the
denominator is higher than the exponent in the numerator, and thus the expression tends to $0$ as $q \rightarrow \infty$, as asserted.

\mn \textbf{Case 2: $\F$ contains a sub-family $S$ of size $p - q^{\frac{(1-\lambda)(d^{k+1}-1)}{d^k(d-1)}}$ without
an intersecting $(q-q^{1-\lambda/2})$-tuple, for $\lambda$ determined in Case 1.} Denote the maximal size of an intersecting sub-family
of $S$ by $q-q^{1-\lambda/2}-t$, for $t\geq 1$. In such a case, $\F \setminus S$ satisfies the
\[
\left(q^{\frac{(1-\lambda)(d^{k+1}-1)}{d^k(d-1)}},q^{1-\lambda/2}+t \right)
\]
property. It is easy to see that the pair $\left(q^{\frac{(1-\lambda)(d^{k+1}-1)}{d^k(d-1)}},q^{1-\lambda/2}+t \right)$ is $(k,\lambda/4)$-close. Hence, by the induction hypothesis, $\F \setminus S$ can be pierced by
\[
q^{\frac{(1-\lambda)(d^{k+1}-1)}{d^k(d-1)}} - q^{1-\lambda/2} - t +k+1
\]
points. As $S$ can clearly be pierced by
\[
(p - q^{\frac{(1-\lambda)(d^{k+1}-1)}{d^k(d-1)}}) - (q-q^{1-\lambda/2}-t)+1
\]
points (by piercing its maximal intersecting subfamily by a single point and each other element by a
separate point), $\F$ has a transversal of size
\[
\left(q^{\frac{(1-\lambda)(d^{k+1}-1)}{d^k(d-1)}} - q^{1-\lambda/2} - t +k+1 \right) + \left(
(p - q^{\frac{(1-\lambda)(d^{k+1}-1)}{d^k(d-1)}}) - (q-q^{1-\lambda/2}-t)+1 \right) = p-q + (k+2).
\]
This completes the inductive proof.
\end{proof}


\medskip

Now we are ready to complete the proof of part (c) of Theorems~\ref{Thm:Main} and \ref{Thm:Main2}. Let us recall the formulation of the result.

\medskip

\noindent \textbf{Theorem 1.3(c)} For any $\eps>0$, there exists $p_0(\eps,d)$ such that for all $p>p_0$ and all $q \geq p^{(d-1)/d+\eps}$, we have
$p-q+1 \leq \HD_d(p,q) \leq p-q+2$.

\medskip

\begin{proof}
The proof repeats the argument of Proposition~\ref{Prop:Large-q}, using Proposition~\ref{Prop:Semi-Large-q} instead of Theorem~\ref{Thm:Main}(b).
We present the required changes, referring to the proof of Proposition~\ref{Prop:Large-q} where no changes are needed.

\mn It is well-known that $\HD_d(p,q) \geq p-q+1$ for all $(p,q)$ (cf.~\cite{HD57}). Hence, we only have to show $\HD_d(p,q) \leq p-q+2$.
Let $p,q$ satisfy the assumptions (with $p_0$ to be defined below), and let $\F$ be a family of compact convex sets
in $\Re^d$ that satisfies the $(p,q)$-property.

\mn \textbf{Case 1: $\F$ satisfies the $(p - \lfloor \frac{q}{d-1} \rfloor, (d-1) \lceil \frac{q}{d} \rceil -d)$ property.}
Let $k$ be the unique integer such that
\begin{equation}\label{Eq:Semi-large5}
\frac{d-1}{d(d^{k+1}-1)} < \eps \leq \frac{d-1}{d(d^{k}-1)}.
\end{equation}
(We note that $k \approx \log_d(1/\eps)$.) Denoting $\eps'= \eps-\frac{d-1}{d(d^{k+1}-1)}$, we have
\[
q \geq p^{\frac{d-1}{d}+\eps} = p^{\frac{d-1}{d}+\frac{d-1}{d(d^{k+1}-1)}+\eps'}= p^{\frac{d^k(d-1)}{d^{k+1}-1}+\eps'},
\]
and thus, for a sufficiently large $p$ (as function of $d,\eps'$),
\[
(d-1) \lceil \frac{q}{d} \rceil -d \geq \frac{q}{3} \geq p^{\frac{d^k(d-1)}{d^{k+1}-1}+\eps'/2} \geq (p - \lfloor \frac{q}{d-1} \rfloor)^{\frac{d^k(d-1)}{d^{k+1}-1}+\eps'/2}.
\]
Hence, by Proposition~\ref{Prop:Semi-Large-q}, for $p>p_3(\eps,k,d)$, $\F$ has a transversal of size
\begin{equation}\label{Eq:Aux1.5}
(p - \lfloor \frac{q}{d-1} \rfloor) - ((d-1) \lceil \frac{q}{d} \rceil -d) + (k+1).
\end{equation}
Therefore, similarly to the proof of Proposition~\ref{Prop:Large-q}, if we show that the term $k+1$ in~\eqref{Eq:Aux1.5} is negligible with respect to $\frac{q}{d(d-1)}$, it will follow that for a sufficiently large $p$, $\F$ has a transversal of size less than $p-q$. This clearly holds, as $k$
depends only on $\eps$.

\mn \textbf{Case 2: $\F$ contains a sub-family $S$ of size $p - \lfloor \frac{q}{d-1} \rfloor$ without
an intersecting $((d-1) \lceil \frac{q}{d} \rceil -d)$-tuple.} The proof that in this case, $\F$ has a transversal of size
at most $p-q+2$, is exactly the same as in Proposition~\ref{Prop:Large-q}.
\end{proof}

\section{A $(p,2)$-theorem in the Plane for Sets with Union Complexity Below a Quadratic Bound}
\label{sec:(p,2)}

\subsection{Proof of Theorem~\ref{Thm:(p,2)}}

We start with a definition and two lemmas.
\begin{definition}
We call a finite family $\F$ of sets {\em exactly $2$-intersecting} if it is pairwise intersecting but no $3$ sets from $\F$ have a common element.
\end{definition}

\begin{lemma}\label{p2convex}
Let $\F$ be a family of compact convex sets in the plane satisfying the $(p,2)$ property and having no exactly $2$-intersecting subfamily of size $k$. Then $\F$ satisfies the $(p^4k,3)$-property and thus has a transversal of size $O(k^4 p^{16})$.
\end{lemma}

\begin{proof}
The proof combines Theorem~\ref{Thm:Main2}(a) and a Ramsey argument for the intersection graph of convex sets.
Let $R(i,j)$ be the minimum integer $R$ such that any family of $R$ convex sets either has a pairwise intersecting subset of cardinality $i$ or a pairwise disjoint subset of cardinality $j$. Larman et al. \cite{PT94} proved that $R(i,j) \leq i j^4$. We show that $\F$ has the $(R(k,p),3)$-property and hence admits a transversal of size $\HD_2(R(k,p),3)=O((R(k,p))^4)= O(k^4p^{16})$ by Theorem~\ref{Thm:Main2}(a). Indeed, consider a subfamily $\F' \subset \F$ of $R(k,p)$ sets. By the definition of $R$, $\F'$ must contain either a family  of $p$ sets such that no pair of them intersect or a family $\cal S$ of $k$ sets such that every pair in $\cal S$ intersect. The former cannot happen by our $(p,2)$ assumption for $\F$. Hence, there exists a pairwise intersecting family $\cal S \subset \F'$ of size $k$. By our assumption, $\cal S$ is not exactly $2$-intersecting, so we find three intersecting elements of $\cal S$. This completes the proof of the lemma.
\end{proof}

\begin{lemma}\label{unioncomplexity}
The union complexity of $k\ge3$ exactly $2$-intersecting compact convex sets in the plane is at least $\binom k2$.
\end{lemma}

\begin{proof}
We claim that the boundaries of each pair of our sets intersect and that all these intersection points are on the boundary of the union. The first assertion holds as the sets are pairwise intersecting and no set is contained in another one (we use our assumption that $k\ge3$ here). The second assertion holds as no three of our sets intersect, so the intersections of the boundaries must lie outside all the other sets.
\end{proof}

Note that the lower bound in Lemma~\ref{unioncomplexity} is tight for pairwise intersecting line segments in general position. Using a recent result of Pach et al.~\cite{PRT} one can improve the $\binom k2$ lower bound in the lemma to $(2-o(1))k^2$ for convex sets which are {\em not} line segments. Consequently, the condition for the union complexity in Theorem~\ref{Thm:(p,2)} can be made similarly weaker for sets with nonempty interior.

\medskip

Now we are ready to present the proof of Theorem~\ref{Thm:(p,2)}.

\begin{proof}[Proof of Theorem~\ref{Thm:(p,2)}]
Let $\F$ be a family that satisfies the assertion of the theorem. By Lemma~\ref{unioncomplexity}, $\F$ does not contain an exactly $2$-intersecting subfamily of size $k$. Hence, by Lemma~\ref{p2convex}, $\F$ has a transversal of size $O(k^4 p^{16})$. This completes the proof.
\end{proof}

\subsection{A constant factor approximation algorithm of the max-clique for families with bounded union complexity}

We need several standard computational assumptions on the family $\F$, such as: Computing the intersection points of any pair of boundaries of elements in $\F$ can be done in constant time, etc.

The algorithm is very simple:
\medskip

\noindent
\bfm{Max-Clique $\mbs{\F}$:}\\
\bfm{Input:} A finite family $\F$ with sub-quadratic union complexity\\
\bfm{Output:} A subset $\F' \subset \F$ of pairwise intersecting sets.
\begin{algorithmic}[1]
    \STATE Compute the arrangement $\A(\F)$.
    \STATE{For every cell $\sigma$ in $\A(\F)$ find the subset $\F_\sigma \subset \F$ of sets in $\F$ containing $\sigma$.}
        \STATE  Let $\sigma_0$ be the cell for which $\cardin{\F_{\sigma_0}}$ is maximal.
    \STATE Return $\F_{\sigma_0}$.
\end{algorithmic}
\medskip

Clearly, the output of the algorithm is indeed a clique in the intersection graph of $\F$. To assess the performance of the algorithm, let $S$ be the largest clique in this graph. Clearly, $S$ is a pairwise intersecting family (i.e., it satisfies the $(2,2)$-property), and thus has a constant size transversal $T$. One of the points of $T$ must be contained in at least $|S|/|T|$ sets of $S$, so for the corresponding cell $\sigma$ we have $|\F_\sigma|\ge|S|/|T|$. The algorithm outputs the family $\F_{\sigma_0}$ with $|\F_{\sigma_0}|\ge|\F_\sigma|$ so it provides a constant $1/|T|$-approximation of the size of the maximal clique.

\section{Discussion}
\label{sec:discussion}

To put Theorem~\ref{Thm:Main} into context, we compare it with the previously known results in each of the ranges of $q$.

\mn For a very large $q$, Theorem~\ref{Thm:Main}(c) is almost tight, leaving only two possible values for $\HD_d(p,q)$. As mentioned in the introduction, this is the first ``good'' estimate of $\HD_d(p,q)$ for any $(p,q)$ outside of the range covered by the Hadwiger-Debrunner theorem. It would be interesting to verify which of the two cases $p-q+1$ or $p-q+2$ is the correct answer.

\mn For a constant $q$, our upper bound (i.e., $\tilde{O}(p^{d\left(\frac{q-1}{q-d}\right)})$) improves over the Alon-Kleitman $\tilde{O}(p^{d^2+d})$ bound already for $d=2,q=3$ (yielding $O(p^4)$ instead of $O(p^6)$). When $q$ is a very large constant, the exponent in our bound tends to $d$. Likewise, when $\log p \leq q \leq p^{(d-1)/d}$, our bound is $\tilde{O}((p/q)^d)$.

It is worth noting that one cannot prove a bound of the type $\HD_d(p,q) = O((p/q)^d)$ (and in particular, $\HD_d(p,q) = O(p^d)$ for a fixed $q$) without improving the bounds on $f(\eps,d)$ for weak $\eps$-nets. Indeed, fix $d$ and $\eps$. Assume for simplicity that $\eps = 1/r$ for some integer $r$ and put $p = rq+1$. We claim that
\begin{equation}\label{Eq:Aux2}
f(1/r,d) \leq \HD_d(p,q).
\end{equation}
To see this, let $S$ be a finite set in $\Re^d$ and let $\F$ be the family of all convex sets containing at least $\frac{\cardin{S}}{r}$ points of $P$. Note that any transversal for $\F$ is a weak $(1/r)$-net for $S$. It is easily seen that $\F$ satisfies the $(p,q)$-property and thus admits a transversal of size $\HD_d(p,q)$, which implies~\eqref{Eq:Aux2}.

\subparagraph*{Acknowledgements}

\noindent We wish to thank Noga Alon and Andreas Holmsen for helpful comments.






\bibliographystyle{alpha}
\bibliography{references}


\end{document}